\documentclass{amsart}
\usepackage[utf8]{inputenc}
\usepackage[english]{babel}
\usepackage{amssymb,upref,calc,url,amsmath,amscd,amsthm,verbatim} % calrsfs,cite
\usepackage[mathscr]{eucal}
\usepackage{graphicx}
\usepackage[all]{xy}
\usepackage[in]{fullpage}

\theoremstyle{theorem}
\newtheorem{thm}{Theorem}

\newtheorem{lem}[thm]{Lemma}

\theoremstyle{remark}

\theoremstyle{definition}

\begin{document}

\title{Classification of $\mathbb{C}P^2$-multiplicative Hirzebruch genera}
\author{Victor M. Buchstaber, Elena Yu. Netay}
\date{}

\address{Steklov Mathematical Institute, Russian Academy of Sciences, Gubkina str. 8, 119991, Moscow, Russia.}
\email{bunkova@mi.ras.ru (E.Yu.Netay), buchstab@mi.ras.ru (V.M.Buchstaber).}

\maketitle

\begin{abstract}
The short article \cite{BN} states results on $\mathbb{C}P(2)$-multiplicative Hirzebruch genera. The aim of the following text is to give a proof of Theorem 3 from \cite{BN}. This proof uses only the technique of functional differential equations.
\end{abstract}

\section{Preliminaries} \text{}

	Let $R$ be a commutative torsion-free ring with unity and no zero divisors, and let $L_f: \Omega_U \to R$ be the Hirzebruch genus determined by the series
$f(x) =  x + \sum_{k = 1}^\infty f_k {x^{k+1} \over (k+1)!}$, where $f_k \in R$.

A Hirzebruch genus $L_f: \Omega_U \to R$ is called \emph{$\mathbb{C}P(2)$-multiplicative}, if we have
$L_f[M] = L_f[\mathbb{C}P(2)] L_f[B]$ for any bundle of stably complex manifolds $M \to B$ with fiber $\mathbb{C}P(2)$ and structure group $G$ such that $U^*(B G)$ is torsion-free. From the localization theorem for the universal toric genus (see \cite{BPR}) for the standard action of the torus $T^3$ on the complex projective plane $\mathbb{C}P(2)$, theorem holds:

\begin{thm} \label{t0}
A genus $L_f$ is $\mathbb{C}P(2)$-multiplicative if and only if $f(x)$ solves the functional equation
\begin{equation} \label{fe}
	{ 1 \over f(t_1 - t_2) f(t_1 - t_3)} + {1 \over f(t_2 - t_1) f(t_2 - t_3)} + {1 \over f(t_3 - t_1) f(t_3 - t_2)} = C, \quad C \in R.
\end{equation}
\end{thm}

In~\cite{H} it was shown with the help of equation \eqref{fe} that for bundles of \emph{oriented} manifolds the universal $\mathbb{C}P(2)$--multiplicative genus is determined by the signature of the manifold.
We~have~$C = L_f[\mathbb{C}P(2)] = {3 f_1^2- f_2 \over 2}$.

\section{Theorem} \text{}

In \cite{BN} the following theorem is proposed. Its proof is given in the next section.

\begin{thm} \label{t1} Let $L_f$ be a $\mathbb{C}P(2)$-multiplicative genus.

	If $L_f[\mathbb{C}P(2)] \ne 0$, then $L_f$ is the two-parametric Todd genus, and
\begin{equation} \label{ft}
	f(x) = {e^{\alpha x} - e^{\beta x} \over \alpha e^{\alpha x} - \beta e^{\beta x}} , \quad  f_1 = - (\alpha + \beta), \quad f_2 = 2 \alpha \beta + f_1^2, \quad f_3 = 4 f_1 f_2 - 3 f_1^3.
\end{equation}

	If $L_f[\mathbb{C}P(2)] = 0$, then it is a two-parametric case of general elliptic genus in the terminology of \cite{Buc10},~and
\begin{equation} \label{fp}
	f(x) = - { 2 \wp(x) + {a^2 \over 2} \over \wp'(x) - a \wp(x) + b - {a^3 \over 4}}.
\end{equation}
Here $\wp$ and $\wp'$ are Weierstrass functions of the elliptic curve with parameters
$g_2 = - {1 \over 4} (8 b - 3 a^3) a$, $g_3 = {1 \over 24} (8 b^2 - 12 a^3 b + 3 a^6)$, and discriminant
$\Delta = - b^3 (3 b - a^3)$.
The parameters $a$ and $b$ are related to the coefficients of the series $f(x)$ by
$f_1 = - a$, $f_2 = 3 a^2$, $f_3 = 12 b - 9 a^3$.
\end{thm}

The genus determined by $f(x)$ as in \eqref{fp} was first introduced in \cite{BucNet11}.

\section{Proof} \text{}

The proof of the theorem follows as a compilation of theorem \ref{t0} with the following three lemmas, each given with its own proof.

For convenience set $q(x) = {1 \over f(x)}$.
Denote $x = t_1 - t_2$, $y = t_2 - t_3$. Equation \eqref{fe} takes the form
\begin{equation} \label{q2}
	q(x)  q(x+y) +  q(- x)  q(y) + q(- x - y)  q(- y) = C.
\end{equation}

\begin{lem} \label{lem1}
	The function $q(x) = {1 \over f(x)}$, where
	\[
		f(x) = {e^{\alpha x} - e^{\beta x} \over \alpha e^{\alpha x} - \beta e^{\beta x}} 
	\]
satisfies the functional equation \eqref{q2} for $C = \alpha^2+ \alpha \beta + \beta^2$.
\end{lem}

\begin{proof}
	The proof is a straightforward substitution, namely, equation \eqref{q2} takes the form
\[
	{(\alpha e^{\alpha x} - \beta e^{\beta x}) (\alpha e^{\alpha (x+y)} - \beta e^{\beta (x+y)}) \over (e^{\alpha x} - e^{\beta x}) (e^{\alpha (x+y)} - e^{\beta (x+y)}) } + {(\alpha e^{-\alpha x} - \beta e^{-\beta x}) (\alpha e^{\alpha y} - \beta e^{\beta y}) \over ( e^{-\alpha x} - e^{-\beta x}) (e^{\alpha y} - e^{\beta y}) } + {(\alpha e^{- \alpha (x+y)} - \beta e^{- \beta (x+y)}) (\alpha e^{-\alpha y} - \beta e^{-\beta y}) \over ( e^{- \alpha (x+y)} - e^{- \beta (x+y)}) ( e^{-\alpha y} - e^{-\beta y})} = C,
\]
which after multiplication of the nominator and denominator by a relevant factor becomes
\[
	{(\alpha e^{\alpha x} - \beta e^{\beta x}) (\alpha e^{\alpha (x+y)} - \beta e^{\beta (x+y)}) \over (e^{\alpha x} - e^{\beta x}) (e^{\alpha (x+y)} - e^{\beta (x+y)}) } + {(\beta e^{\alpha x} - \alpha e^{\beta x} ) (\alpha e^{\alpha y} - \beta e^{\beta y}) \over (e^{\alpha x} - e^{\beta x} ) (e^{\alpha y} - e^{\beta y}) } + {(\beta e^{\alpha (x+y)} - \alpha e^{\beta (x+y)}) (\beta e^{\alpha y} - \alpha e^{\beta y}) \over (e^{\alpha (x+y)} - e^{\beta (x+y)} ) (e^{\alpha y} - e^{\beta y})} = C.
\]
Bringing to a common factor one gets
\begin{multline*}
	(\alpha e^{\alpha x} - \beta e^{\beta x}) (e^{\alpha y} - e^{\beta y}) (\alpha e^{\alpha (x+y)} - \beta e^{\beta (x+y)}) 
	+ (\beta e^{\alpha x} - \alpha e^{\beta x}) (\alpha e^{\alpha y} - \beta e^{\beta y}) (e^{\alpha (x+y)} - e^{\beta (x+y)})  + \\
	+ (e^{\alpha x} - e^{\beta x}) (\beta e^{\alpha y} - \alpha e^{\beta y}) (\beta e^{\alpha (x+y)} - \alpha e^{\beta (x+y)})  = C (e^{\alpha x} - e^{\beta x}) (e^{\alpha y} - e^{\beta y}) (e^{\alpha (x+y)} - e^{\beta (x+y)}).
\end{multline*}
Now this expression is available for term-by-term check, like at $e^{2 \alpha (x+y)}$ we have
\[
	\alpha^2  + \alpha \beta + \beta^2  = C 
\]
and the same for all other coefficients.
\end{proof}

\begin{lem} \label{lem2}
	The function $q(x) = {1 \over f(x)}$, where
\[
	f(x) = - { 2 \wp(x) + {a^2 \over 2} \over \wp'(x) - a \wp(x) + b - {a^3 \over 4}}
\]
with parameters $g_2 = - {1 \over 4} (8 b - 3 a^3) a$ and $g_3 = {1 \over 24} (8 b^2 - 12 a^3 b + 3 a^6)$ of the Weierstrass $\wp$-function satisfies the functional equation \eqref{q2} for $C = 0$.
\end{lem}

\begin{proof}
We have
\[
	q(x) =  {a \over 2} - { b \over 2 \wp(x) + {a^2 \over 2}} - { \wp'(x) \over 2 \wp(x) + {a^2 \over 2}}.
\]
For $C=0$ equation \eqref{q2} after the substitution of $q(x)$ takes the form (here we take into account that $\wp$ is an even function and $\wp'$ is odd)
\begin{multline*}
	\left( {a \over 2} - { b \over 2 \wp(x) + {a^2 \over 2}} - { \wp'(x) \over 2 \wp(x) + {a^2 \over 2}}\right)
	\left( {a \over 2} - { b \over 2 \wp(x+y) + {a^2 \over 2}} - { \wp'(x+y) \over 2 \wp(x+y) + {a^2 \over 2}}\right) +\\
	+ \left( {a \over 2} - { b \over 2 \wp(x) + {a^2 \over 2}} + { \wp'(x) \over 2 \wp(x) + {a^2 \over 2}}\right)
	\left( {a \over 2} - { b \over 2 \wp(y) + {a^2 \over 2}} - { \wp'(y) \over 2 \wp(y) + {a^2 \over 2}}\right) + \\
	+ \left( {a \over 2} - { b \over 2 \wp(x+y) + {a^2 \over 2}} + { \wp'(x+y) \over 2 \wp(x+y) + {a^2 \over 2}}\right)
	\left( {a \over 2} - { b \over 2 \wp(y) + {a^2 \over 2}} + { \wp'(y) \over 2 \wp(y) + {a^2 \over 2}}\right) = 0.
\end{multline*}
After bringing this expression to a common denominator we obtain that it is required to prove the relation
\begin{multline*}
	\left(\wp(y) + {a^2 \over 4}\right) \left(\wp'(x) + b - a \left(\wp(x) + {a^2 \over 4}\right)\right) \left(\wp'(x+y) + b - a \left(\wp(x+y) + {a^2 \over 4})\right)\right) - \\
	- \left(\wp(x+y) + {a^2 \over 4}\right) \left(\wp'(x) - b + a \left(\wp(x) + {a^2 \over 4}\right)\right) \left(\wp'(y) + b - a \left(\wp(y) + {a^2 \over 4}\right)\right) + \\
	+ \left(\wp(x) + {a^2 \over 4}\right) \left(\wp'(x+y) - b + a  \left(\wp(x+y) + {a^2 \over 4}\right)\right) \left(\wp'(y) - b + a \left(\wp(y) + {a^2 \over 4}\right)\right) = 0.
\end{multline*}
Consider the left part as a function of $x$ where $y$ is a parameter. It is a two-periodic function, it might have poles only in points comparable to $x = 0$ and $x = -y$. Consider this function for $x = 0$. We obtain $0$ at ${1 \over x^3}$, while at ${1 \over x^2}$ we obtain
\[
	32 (3 a^4 - 8 a b - 4 g_2) \wp(y)+ 12 a^6 - 64 a^3 b + 16 a^2 g_2 - 192 g_3 + 64 b^2, 
\]
which gives $0$ after substituting $g_2$ and $g_3$. At ${1 \over x}$ we get
\[
	16 (3 a^4 - 8 a b - 4 g_2) \wp'(y)
\]
which gives $0$ again. Therefore the left part has no poles in points comparable to $x = 0$. As the equation is invariant under substitutions $(x \to y, y \to x, a \to -a, b \to -b)$ and $x \to y, y \to  - x - y, a \to a, b \to b$, thus it has no poles comparable to $x = -y$. Therefore the left part of the expression, being a meromorphic function without poles, must be constant. Calculation of the free term at $x = 0$ shows that this expression is equal to zero.
\end{proof}

\begin{lem}
	The functional equation \eqref{q2} does not have solutions other then stated in Lemmas \ref{lem1} and \ref{lem2}.
\end{lem}

\begin{proof}
The series decomposition of \eqref{q2} in $y$ taking into account the initial conditions gives at $y^k$ for $k = 0$ the equation
\begin{equation} \label{u1}
	q(x)^2 - f_1 q(-x) + q'(-x) = C,
\end{equation}
for $k=1$ the derivative of \eqref{u1} and for $k=2$ the equation
\begin{equation} \label{u2}
	6 q(x) q''(x) - K q(-x) + (3 f_1^2 - 2 f_2) q'(-x) - 3 f_1 q''(-x) + 2 q'''(-x) = 0, \qquad K = 3 f_1^3 - 4 f_1 f_2 + f_3.
\end{equation}
Decomposing the equations \eqref{u1} and \eqref{u2} and taking into account initial conditions, we obtain from \eqref{u1} at $x^0$ the relation $2 C = 3 f_1^2 - f_2$. Further from $x^k$ in \eqref{u1} and $x^{k-2}$ in \eqref{u2} for $k = 2, 3, 4, 5$ we get coinciding relations
\begin{align*}
	f_4 &= 15 f_1^4 - 25 f_1^2 f_2 + 7 f_1 f_3 + 4 f_2^2,\\
	f_5 &=15 f_1^3 f_2 - 15 f_1^2 f_3 - 10 f_1 f_2^2 + 6 f_1 f_4 + 5 f_2 f_3,\\
	2 f_6 &= 315 f_1^6 - 945 f_1^4 f_2 + 345 f_1^3 f_3 + 660 f_1^2 f_2^2 - 93 f_1^2 f_4 - 290 f_1 f_2 f_3 - 60 f_2^3 + 18 f_1 f_5 + 32 f_2 f_4 + 20 f_3^2,\\
	f_7 &= 210 f_1^5 f_2 - 210 f_1^4 f_3 - 420 f_1^3 f_2^2 + 105 f_1^3 f_4 + 420 f_1^2 f_2 f_3 + 140 f_1 f_2^3 - 35 f_1^2 f_5 - 112 f_1 f_2 f_4 - 70 f_1 f_3^2 - \\ & - 70 f_2^2 f_3 + 8 f_1 f_6 + 14 f_2 f_5 + 21 f_3 f_4.
\end{align*}
At $x^6$ in \eqref{u1} and at $x^{4}$ in \eqref{u2} we get different relations:
\begin{align*}	
	3 f_8 &= 8505 f_1^8 - 36855 f_1^6 f_2 + 14805 f_1^5 f_3 + 48300 f_1^4 f_2^2 - 4599 f_1^4 f_4 - 29820 f_1^3 f_2 f_3 - 19320 f_1^2 f_2^3 + 1134 f_1^3 f_5 + \\ & 
	+ 6552 f_1^2 f_2 f_4 + 4095 f_1^2 f_3^2 + 10500 f_1 f_2^2 f_3 + 1120 f_2^4 - 222 f_1^2 f_6 - 980 f_1 f_2 f_5 - 1470 f_1 f_3 f_4 - 924 f_2^2 f_4 - \\ & - 1155 f_2 f_3^2 + 33 f_1 f_7 + 80f_2 f_6 + 140 f_3 f_5 + 84 f_4^2,\\
	19 f_8 &= 53865 f_1^8 - 233415 f_1^6 f_2 + 94500 f_1^5 f_3 + 304920 f_1^4 f_2^2 - 29862 f_1^4 f_4 - 188370 f_1^3 f_2 f_3 - 121380 f_1^2 f_2^3 + \\ &
	+ 7497 f_1^3 f_5 + 41706 f_1^2 f_2 f_4 + 25515 f_1^2 f_3^2 + 65730 f_1 f_2^2 f_3 + 7000 f_2^4 - 1476 f_1^2 f_6 - 6300 f_1 f_2 f_5 - \\& 
	- 9072 f_1 f_3 f_4 - 5796 f_2^2 f_4 - 7140 f_2 f_3^2 + 216 f_1 f_7 + 516 f_2 f_6 + 861 f_3 f_5 + 504 f_4^2.
\end{align*}
Comparing this expressions for $f_8$ taking into account the expressions for $f_4, f_5, f_6, f_7$ and $C$ we obtain the relation
\[
	C K^2 = 0.
\]
Therefore we obtain either $C = 0$, which gives solution \eqref{fp}, or $K = 0$, which gives solution \eqref{ft}.

From \eqref{u1} and initial conditions it follows that for given $f_1$, $f_2$, $f_3$ all $f_k$ for $k \geqslant 4$ are uniquely defined, thus there are no other solutions.
\end{proof}

\end{document}